\documentclass[11pt]{article}
\usepackage[english]{babel}
\usepackage[a4paper,scale=0.75,hcentering]{geometry}
\usepackage{amsmath,amsthm,amssymb,mathrsfs,mathtools}
\usepackage[hyperfootnotes=false]{hyperref}
\hypersetup{colorlinks=true,linkcolor=blue,urlcolor=red,citecolor=blue}
\allowdisplaybreaks
\newcounter{results}[section]

\theoremstyle{plain}
\newtheorem{theorem}[results]{Theorem}
\newtheorem{lemma}[results]{Lemma}
\newtheorem{proposition}[results]{Proposition}
\theoremstyle{remark}
\newtheorem{remark}[results]{Remark}
\numberwithin{equation}{section}
\newcommand{\R}{\mathbb R}
\newcommand{\de}{\,\mathrm d}
\newcommand{\Haus}{\mathscr H}
\DeclareMathOperator{\Tr}{Tr}
\DeclareMathOperator{\tr}{tr}
\DeclareMathOperator{\dist}{dist}
\DeclareMathOperator{\diver}{div}
\begin{document}

\title{\bf Extremals for Sharp Trace Gagliardo--Nirenberg--Sobolev Inequalities on Convex Cones\thanks{This work is supported by National Key R$\&$D Program of China (Grant 2023YFA1010001) and NSFC(12571123). E-mail addresses: zhou-yx22@mails.tsinghua.edu.cn(Zhou), zhangyf22@mails.tsinghua.edu.cn(Zhang), zou-wm@mail.tsinghua.edu.cn (Zou)}}
\author{{\bf Yuxuan Zhou, Yingfang Zhang, Wenming Zou}\\ {\footnotesize \it Department of Mathematical Sciences, Tsinghua University, Beijing 100084, China.}}
\date{}
\maketitle

\begin{abstract}
{\small
Zugmeyer~\cite{Zug} established sharp trace Gagliardo--Nirenberg--Sobolev inequalities on convex cones for $a\geq n>p>1$. We classify all nonzero real-valued extremals of these inequalities. Let the cone be the epigraph of a finite convex function that is positively homogeneous of degree one, and let $H$ be any norm. We prove that equality holds precisely for functions of the form
$A H(x+te+\ell)^{-(a-p)/(p-1)}$, where $A\in\R\setminus\{0\}$, $t>0$, $e=(0,\ldots,0,1)$, and $\ell$ belongs to the largest vector subspace contained in the cone. The result includes the critical Sobolev trace inequality at $a=n$, as well as nonsmooth norms and cones, half-spaces, and cones containing lines.
\vskip0.1in
\noindent{\bf Key words:} Sharp trace inequalities, Gagliardo--Nirenberg--Sobolev inequalities, extremal functions, convex cones, anisotropic norms.
\vskip0.1in
\noindent{\bf 2020 Mathematics Subject Classification:} Primary 46E35; Secondary 26D10.
}
\end{abstract}

\section{Introduction}\label{sec:introduction}

The sharp Sobolev trace inequality on a half-space bounds the $L^{p_\partial^*}$ norm of the boundary trace in terms of the $L^p$ norm of the gradient. Let $n>p>1$, write $\R^n_+=\{(x',x_n):x_n>0\}$, and set
\begin{equation}\label{eq:historical-exponents}
 p^*=\frac{np}{n-p},\qquad p_\partial^*=\frac{p(n-1)}{n-p}\,.
\end{equation}
The inequality has the form
\begin{equation}\label{eq:halfspace-trace}
 \left(\int_{\R^{n-1}}|\Tr u(x',0)|^{p_\partial^*}\de x'\right)^{1/p_\partial^*}
 \leq S_{n,p}^{\mathrm{tr}}\left(\int_{\R^n_+}|\nabla u|^p\de x\right)^{1/p}\,.
\end{equation}
For $p=2$ and $n\geq3$, Escobar~\cite[Theorem 1]{Escobar} determined the sharp constant in~\eqref{eq:halfspace-trace} and characterized its equality cases in a class of smooth, sufficiently decaying functions. For $n\geq3$ and $1<p<n$, Nazaret~\cite[Theorem 1]{Nazaret} established the sharp inequality for arbitrary norms, with the dual norm in the gradient integral, and exhibited explicit equality cases. Maggi and Neumayer~\cite[Theorem 1.1]{MaggiNeumayer} subsequently characterized the minimizers of the family of problems
\begin{equation}\label{eq:bridge}
 \mathscr S(\tau)=\inf\left\{\|\nabla u\|_{L^p(\R^n_+)}:
 \|u\|_{L^{p^*}(\R^n_+)}=1,\quad
 \|\Tr u\|_{L^{p_\partial^*}(\R^{n-1})}=\tau\right\},\qquad\tau>0\,,
\end{equation}
for $n\geq2$ and $1<p<n$. This family connects the classical Sobolev and trace inequalities. For the Euclidean trace inequality~\eqref{eq:halfspace-trace}, the nonzero nonnegative extremals have the form
\begin{equation}\label{eq:halfspace-extremals}
 u(x',x_n)=A\bigl(|x'-z'|^2+(x_n+t)^2\bigr)^{-\frac{n-p}{2(p-1)}},
 \qquad A>0,\quad t>0,\quad z'\in\R^{n-1}\,.
\end{equation}

Related sharp inequalities have been studied in anisotropic and weighted settings. Cordero-Erausquin, Nazaret and Villani~\cite{CNV} established sharp Sobolev and Gagliardo--Nirenberg inequalities for arbitrary norms in the whole space. Cabr\'e, Ros-Oton and Serra~\cite[Theorem 1.3 and Section 1.3]{CabreRosOtonSerra} proved sharp weighted isoperimetric inequalities in convex cones and obtained corresponding Sobolev inequalities for homogeneous weights satisfying a concavity condition. These inequalities involve norms over the domain. For boundary traces on convex cones, Zugmeyer~\cite[Theorem 1.2 and Corollary 1.3]{Zug} established a sharp family that includes both Sobolev trace and Gagliardo--Nirenberg--Sobolev trace inequalities.

Let $\varphi:\R^{n-1}\to\R$ be finite, convex and positively homogeneous of degree one. Write
\begin{equation*}
 \Omega=\{(x',x_n):x_n\geq\varphi(x')\},\qquad U=\Omega^\circ,\qquad e=(0,\ldots,0,1),\qquad L=\Omega\cap(-\Omega)\,.
\end{equation*}
The subspace $L$ is called the lineality space of $\Omega$. Let $H$ be an arbitrary norm and $H^\circ(\xi)=\sup_{H(y)\leq1}\xi\cdot y$ be its dual norm. Let $a\geq n$, and set
\begin{equation}\label{eq:parameters}
 p'=\frac p{p-1},\qquad q=\frac{p(a-1)}{a-p},\qquad \theta=\frac{a-p}{p(a-n-1)+n}\,.
\end{equation}
For functions in the class
\begin{equation}\label{eq:space}
 X_a=\begin{cases}
 \{f\in W^{1,p}_{\mathrm{loc}}(U): f\in L^q(U),\ \nabla f\in L^p(U)\},&a>n,\\
 \{f\in W^{1,p}_{\mathrm{loc}}(U): f\in L^{p^*}(U),\ \nabla f\in L^p(U)\},&a=n,
 \end{cases}
\end{equation}
write
\begin{equation}\label{eq:functionals}
 E(f)=\int_U H^\circ(\nabla f)^p\de x,\qquad J(f)=\int_{\R^{n-1}}|\Tr f(x',\varphi(x'))|^q\de x'\,.
\end{equation}
In $J(f)$, the measure $\de x'$ is Lebesgue measure in the parameter $x'$ of the boundary graph; we call it the projected boundary measure. When $a>n$, also write $I(f)=\int_U|f|^q\de x$. Zugmeyer's inequality~\cite[Theorem 1.2 and Corollary 1.3]{Zug} is
\begin{equation}\label{eq:sharp}
 J(f)^{1/q}\leq D E(f)^{\theta/p} I(f)^{(1-\theta)/q}\,.
\end{equation}
Here $D=D_{n,p,a,H}(\Omega)$ is the sharp constant. At $a=n$, the exponent $1-\theta$ is zero, and the factor involving $I(f)$ is omitted. The trace is understood in the Sobolev sense. We extend~\eqref{eq:sharp} to all functions in $X_a$ in Section~\ref{sec:preliminaries}.

At $a=n$, inequality~\eqref{eq:sharp} reads
\begin{equation}\label{eq:endpoint-sharp}
 \left(\int_{\R^{n-1}}|\Tr f(x',\varphi(x'))|^{p_\partial^*}\de x'\right)^{1/p_\partial^*}
 \leq D_{n,p,n,H}(\Omega)\left(\int_U H^\circ(\nabla f)^p\de x\right)^{1/p}\,.
\end{equation}
Zugmeyer~\cite[Theorem 1.2]{Zug} proved~\eqref{eq:sharp} with the sharp constant and showed that equality holds for
\begin{equation}\label{eq:zug-model}
 f_0(x)=H(x+e)^{-\frac{a-p}{p-1}}\,.
\end{equation}
His proof uses a strengthened Borell--Brascamp--Lieb inequality and infimal convolution. The function~\eqref{eq:zug-model} attains the sharp constant, but this does not characterize all equality cases. In the present work, we classify all nonzero real-valued extremals in $X_a$.

\begin{theorem}\label{thm:classification}
Let $f\in X_a$ be real-valued and nonzero. Then $f$ attains equality in~\eqref{eq:sharp} if and only if
\begin{equation}\label{eq:family}
 f(x)=A H(x+te+\ell)^{-\frac{a-p}{p-1}}\quad\text{a.e. in }U,
 \qquad A\in\R\setminus\{0\},\quad t>0,\quad\ell\in L\,.
\end{equation}
\end{theorem}

\begin{remark}\label{rem:weighted}
Recently, do {\'O} and Silva~\cite[Theorems 1.2 and 1.3]{doOSilva} proved sharp weighted Sobolev trace inequalities and classified their extremals under compatibility assumptions on the weight, the cone and a distinguished direction. On a convex cone $\Sigma\subset\R^n$, their inequalities compare boundary integrals with respect to $w\de\Haus^{n-1}$ with the weighted gradient energy $\int_\Sigma|\nabla u|^p w\de x$. Our inequality~\eqref{eq:sharp} uses the projected boundary measure $\de x'$, permits an arbitrary norm, and includes the interior norm $\|f\|_{L^q(U)}$ when $a>n$. These results therefore concern different sharp trace inequalities.
\end{remark}

Theorem~\ref{thm:classification} characterizes all nonzero real-valued extremals. The norm need not be smooth or strictly convex. The cone may have a nonsmooth boundary or contain lines; half-spaces are included. The vectors $\ell\in L$ in~\eqref{eq:family} are exactly those for which $\Omega+\ell=\Omega$. For a half-space, these are all vectors tangent to its boundary. The classification concerns functions attaining equality in~\eqref{eq:sharp}.

The proof adapts the analysis of equality cases by Cordero-Erausquin, Nazaret and Villani~\cite[Section 4]{CNV} to the boundary measure $\de x'$. We first use interior variations to prove that every nonzero nonnegative extremal has a strictly positive lower bound on each compact subset of $U$. We then prove a boundary integration formula on the Lipschitz epigraph, including the singular part of the distributional Hessian of the transport potential. The resulting estimate forces the transport map to be affine. Comparing the supports of the measures before and after transport shows that its translation vector has the form $e+\ell$, with $\ell\in L$. To obtain equality in this transport estimate from equality in~\eqref{eq:sharp}, we multiply the extremal by a positive constant and, when $a>n$, rescale its argument. This step uses Young's inequality and the dilations in Zugmeyer's argument~\cite[Section 3.2]{Zug}.

Section~\ref{sec:preliminaries} contains the approximation, regularity and boundary integration results. In Section~\ref{sec:proof}, we prove the transport estimate and the classification theorem.

\section{Preliminary results}\label{sec:preliminaries}

Convexity and positive homogeneity imply that $\varphi$ is subadditive. Since a finite convex function on $\R^{n-1}$ is continuous, $K=\max_{|z'|=1}|\varphi(z')|$ is finite. Then
\begin{equation*}
 \varphi(x')-\varphi(y')\leq\varphi(x'-y')\leq K|x'-y'|\,.
\end{equation*}
Interchanging $x'$ and $y'$ gives $|\varphi(x')-\varphi(y')|\leq K|x'-y'|$. Thus $U$ is a global Lipschitz graph domain. At every differentiability point $x'$ of $\varphi$, Euler's identity gives $\nabla\varphi(x')\cdot x'=\varphi(x')$, while convexity gives
\begin{equation*}
 \varphi(z')\geq\varphi(x')+\nabla\varphi(x')\cdot(z'-x')
 \qquad(z'\in\R^{n-1})\,.
\end{equation*}
Consequently, at almost every graph point $x=(x',\varphi(x'))$, the outward unit normal satisfies
\begin{equation}\label{eq:normal}
 \nu(x)=\frac{(\nabla\varphi(x'),-1)}{\sqrt{1+|\nabla\varphi(x')|^2}},\qquad x\cdot\nu(x)=0,\qquad z\cdot\nu(x)\leq0\quad(z\in\Omega)\,.
\end{equation}
The graph area formula gives
\begin{equation}\label{eq:projected}
 (-e\cdot\nu)\de\Haus^{n-1}=\de x'\,.
\end{equation}
These boundary identities hold $\Haus^{n-1}$-almost everywhere, even when the graph has vertices or edges.

We next construct smooth approximations for $f\in X_a$. Extend $f$ to $\R^n$ by reflection across the graph:
\begin{equation*}
 \mathcal Ef(x',x_n)=\begin{cases}f(x',x_n),&x_n\geq\varphi(x'),\\
 f(x',2\varphi(x')-x_n),&x_n<\varphi(x'),\end{cases}
\end{equation*}
For every finite $t\geq1$ such that $f\in L^t(U)$, this extension satisfies $\|\mathcal Ef\|_{L^t(\R^n)}=2^{1/t}\|f\|_{L^t(U)}$. It also satisfies $\|\nabla\mathcal Ef\|_{L^p(\R^n)}\leq C_{p,K}\|\nabla f\|_{L^p(U)}$. Indeed, $q>p$ when $a>n$, and $p^*>p$ when $a=n$, so $f\in W^{1,p}(U\cap B_R)$ for every $R>0$. Flattening the Lipschitz graph and reflecting across the resulting hyperplane gives the asserted Sobolev extension and gradient bound.

For $a>n$, apply the Euclidean Sobolev inequality~\cite[Theorem 2(ii)]{CNV} to $\min\{(|\mathcal Ef|-\varepsilon)_+,M\}$. This truncation is bounded and is nonzero only on a set of finite measure. It therefore belongs to $L^p(\R^n)$, and its weak gradient is bounded by $|\nabla\mathcal Ef|$. Letting $M\to\infty$ and then $\varepsilon\downarrow0$, we obtain $\mathcal Ef\in L^{p^*}(\R^n)$ and hence $f\in L^{p^*}(U)$. At $a=n$, the same integrability of $\mathcal Ef$ follows directly from $f\in L^{p^*}(U)$. Define
\begin{equation}\label{eq:chainexponents}
 r=\frac p{a-p},\qquad s=\frac{pa}{a-p}=ar=(q-1)p',\qquad q=(a-1)r\,,
\end{equation}
If $a>n$, then $q<s<p^*$, so interpolation between $L^q(U)$ and $L^{p^*}(U)$ gives $f\in L^s(U)$. If $a=n$, then $s=p^*$, and this conclusion is part of the definition of $X_a$.

By applying a cutoff to $\mathcal Ef$ and then mollifying, we obtain $f_j=u_j|_U$, with $u_j\in C_c^\infty(\R^n)$, such that $\nabla f_j\to\nabla f$ in $L^p(U)$. Moreover, $f_j\to f$ in $L^q(U)$ if $a>n$, and in $L^{p^*}(U)$ if $a=n$. The gradient error from a cutoff at radius $R$ is bounded by
\begin{equation*}
 R^{-1}\|\mathcal Ef\|_{L^p(B_{2R}\setminus B_R)}\leq c_n\|\mathcal Ef\|_{L^{p^*}(B_{2R}\setminus B_R)}\to 0\,.
\end{equation*}
For a difference $v=f_j-f_k$, apply Zugmeyer's inequality~\cite[Theorem 1.2 and Corollary 1.3]{Zug} to smooth nonnegative approximations of $|v|$ and pass to the limit. Thus~\eqref{eq:sharp} holds for $v$, and the traces of $f_j$ form a Cauchy sequence in $L^q(\de x')$. On each bounded boundary patch, continuity of the Sobolev trace operator identifies the limit with $\Tr f$. Passing to the limit in~\eqref{eq:sharp} proves the inequality for every $f\in X_a$ and shows that $J(f)<\infty$.

For interior regularity, we use the estimates for scalar quasiminima of variational integrals in Giaquinta and Giusti~\cite{GiaG} and Di Benedetto and Trudinger~\cite{DiT}. Their hypotheses are growth bounds on the integrand, so they apply even when $H$ is not differentiable.

\begin{lemma}\label{lem:regularity}
Let $f\in X_a$ be a nonzero nonnegative extremal of~\eqref{eq:sharp}. Then $f$ has a locally H\"older continuous representative in $U$, with $\min_{\mathcal K} f>0$ for every compact $\mathcal K\Subset U$. Moreover, the function $h=f^q$ satisfies $\nabla h\in L^1(U)$, $\Tr h=(\Tr f)^q$, and
\begin{equation}\label{eq:hspaces}
 h\in L^1(U)\quad(a>n),\qquad h\in L^{n/(n-1)}(U)\quad(a=n)\,.
\end{equation}
\end{lemma}

\begin{proof}
Let $\mathcal L(\xi)=H^\circ(\xi)^p/p$. By equivalence of norms, there are positive constants $c_0,c_1$ such that $c_0|\xi|^p\leq\mathcal L(\xi)\leq c_1|\xi|^p$. Let $v\in X_a$ agree with $f$ outside a compact subset of $U$. Then $\Tr v=\Tr f$. For $a>n$, equality in~\eqref{eq:sharp} for $f$ and the same inequality for $v$ give
\begin{equation*}
 \left(\frac{E(v)}{E(f)}\right)^{\theta/p}
 \left(\frac{I(v)}{I(f)}\right)^{(1-\theta)/q}\geq1\,.
\end{equation*}
Here $E(f)>0$, since zero gradient would make $f$ a constant and integrability would force $f=0$. By the weighted arithmetic--geometric mean inequality, we obtain
$\frac\theta p\frac{E(v)}{E(f)}+
\frac{1-\theta}q\frac{I(v)}{I(f)}\geq
\frac\theta p+\frac{1-\theta}q$.
Hence $f$ locally minimizes
\begin{equation}\label{eq:localminimum}
 \int_{\mathcal O}\left(\mathcal L(\nabla u)+\frac\kappa q|u|^q\right)\de x,\qquad
 \kappa=\frac{(1-\theta)E(f)}{\theta I(f)}\,,
\end{equation}
among functions that agree with $f$ outside a compact subset of $\mathcal O\Subset U$. Competitors $v$ may change sign: since $f\geq0$, the function $|v|$ still agrees with $f$ outside the same compact set. Moreover, $H^\circ(\nabla|v|)=H^\circ(\nabla v)$ almost everywhere, so replacing $v$ by $|v|$ leaves the integral in~\eqref{eq:localminimum} unchanged. At $a=n$, the equality of traces gives $E(v)\geq E(f)$, and $f$ minimizes~\eqref{eq:localminimum} with $\kappa=0$.

For an admissible competitor satisfying $0\leq v\leq f$, we have $\int_{\mathcal O}v^q\leq\int_{\mathcal O}f^q$. Minimality in~\eqref{eq:localminimum} therefore gives $\int_{\mathcal O}\mathcal L(\nabla f)\leq\int_{\mathcal O}\mathcal L(\nabla v)$. For an admissible $v\leq f$ that takes negative values, apply this argument to $v_+=\max\{v,0\}$ and use $\mathcal L(\nabla v_+)\leq\mathcal L(\nabla v)$. It follows that
\begin{equation*}
 \int_{\mathcal O}|\nabla f|^p\de x\leq\frac{c_1}{c_0}\int_{\mathcal O}|\nabla v|^p\de x
 \qquad\text{for every such }v\leq f\,.
\end{equation*}
Thus $f$ is a sub-quasiminimum of the Euclidean $p$-Dirichlet integral in the sense of Di Benedetto and Trudinger, and their local boundedness estimate~\cite[Lemma 4.1 and Corollary 1]{DiT} applies.

Fix $\mathcal O\Subset U$ and choose $m>\operatorname*{ess\,sup}_{\mathcal O} f$. For an admissible competitor $v$, define
\begin{equation*}
 v_m=\max\{-m,\min\{v,m\}\}\,.
\end{equation*}
This truncation is still admissible and has no larger gradient energy. Applying~\eqref{eq:localminimum} to $v_m$ shows that $f$ also minimizes the integral with integrand
\begin{equation*}
 \mathcal F_m(z,\xi)=\mathcal L(\xi)+\frac\kappa q\min\{|z|,m\}^q\,.
\end{equation*}
Since $q>p$,
\begin{equation}\label{eq:regularity-growth}
 c_0|\xi|^p\leq\mathcal F_m(z,\xi)\leq c_1|\xi|^p+\frac\kappa q m^{q-p}|z|^p\,.
\end{equation}
When $z$ is restricted to a bounded interval,~\eqref{eq:regularity-growth} gives the growth bound required by Giaquinta and Giusti~\cite[Theorem 4.2]{GiaG}, with an additive constant depending on that interval. Their theorem yields local H\"older continuity. To obtain Harnack's inequality, divide~\eqref{eq:regularity-growth} by $c_0$. The resulting bound satisfies condition (4.2) of Di Benedetto and Trudinger~\cite[p.\ 305]{DiT}: the coefficient of $|z|^p$ is bounded, and the additive term denoted by $g$ in that condition is zero. Their Harnack inequality~\cite[Corollary 3, p.\ 307]{DiT} therefore applies with no additive error term. Since $U$ is connected and $f\not\equiv0$, $f$ is strictly positive in $U$. Continuity then gives a positive lower bound on each compact subset of $U$.

Finally,~\eqref{eq:chainexponents} gives
\begin{equation*}
 \int_U|\nabla(f^q)|\de x\leq q\|\nabla f\|_{L^p(U)}\|f\|_{L^s(U)}^{q-1}<\infty\,.
\end{equation*}
To justify the chain rule and the trace identity, first replace $f$ by $\min\{f,M\}$ and then let $M\to\infty$, using convergence of the corresponding powers in $W^{1,1}$ on bounded boundary patches. If $a=n$, the identity $qn/(n-1)=p^*$ gives~\eqref{eq:hspaces}. If $a>n$, the required integrability of $h=f^q$ follows from $f\in L^q(U)$.
\end{proof}

We follow Cordero-Erausquin, Nazaret and Villani~\cite[Section 4, Proposition 6 and Lemma 7]{CNV} in treating the distributional Hessian as a matrix-valued measure, including its singular part. The next lemma establishes the boundary integration inequality needed for the measure $\de x'$ on our Lipschitz epigraph. Set $D_1=e+\Omega$.

\begin{lemma}\label{lem:green}
Let $h\geq0$ be continuous in $U$ and belong to $W^{1,1}(U\cap B_R)$ for every $R>0$. Assume that $\nabla h\in L^1(U)$, $\Tr h\in L^1(\de x')$, and either $h\in L^1(U)$ or $h\in L^{n/(n-1)}(U)$. Let $\Phi$ be finite and convex on $U$, with $T=\nabla\Phi\in D_1$ almost everywhere and
\begin{equation}\label{eq:pairassumption}
 \int_U H^\circ(\nabla h)H(T)\de x<\infty\,.
\end{equation}
The distributional Laplacian $\Delta_{\mathcal D'}\Phi$ is a nonnegative Radon measure, and
\begin{equation}\label{eq:green}
 \int_Uh\de(\Delta_{\mathcal D'}\Phi)+\int_{\R^{n-1}}\Tr h\de x'\leq-\int_U\nabla h\cdot T\de x\,.
\end{equation}
In particular, the integral of $h$ against $\Delta_{\mathcal D'}\Phi$ is finite.
\end{lemma}

\begin{proof}
For each $z\in U$, let $\partial\Phi(z)$ denote the subdifferential of $\Phi$ at $z$. Every vector in $\partial\Phi(z)$ lies in $D_1$, since it belongs to the closed convex hull of limits of nearby gradients and $D_1$ is closed and convex. Define an extension to $\R^n$ by taking the supremum of the supporting affine functions at points of $U$:
\begin{equation*}
 \overline\Phi(x)=\sup_{z\in U,\,v\in\partial\Phi(z)}\{\Phi(z)+v\cdot(x-z)\}\,.
\end{equation*}
The function $\overline\Phi$ is lower semicontinuous and convex, never takes the value $-\infty$, and agrees with $\Phi$ on $U$. In particular, it is a proper convex function, although it may equal $+\infty$ outside $U$. For sufficiently large $R$, define
\begin{equation}\label{eq:slopetruncation}
 K_R=D_1\cap\{H\leq R\},\qquad \Phi_R=(\overline\Phi^*+\iota_{K_R})^*,\qquad T_R=\nabla\Phi_R\,.
\end{equation}
Here $F^*(y)=\sup_x\{x\cdot y-F(x)\}$ denotes the convex conjugate, and $\iota_{K_R}$ is zero on $K_R$ and $+\infty$ elsewhere. Thus
\begin{equation*}
 \Phi_R(x)=\sup_{y\in K_R}\{x\cdot y-\overline\Phi^*(y)\}\,.
\end{equation*}
Choose $z_0\in U$ and $v_0\in\partial\Phi(z_0)$. For $R\geq H(v_0)$, we have $v_0\in K_R$ and $\overline\Phi^*(v_0)=z_0\cdot v_0-\Phi(z_0)<\infty$. Evaluating the supremum at $v_0$ gives a finite lower bound. The inequality $\overline\Phi^*(y)\geq z_0\cdot y-\Phi(z_0)$ and compactness of $K_R$ give a finite upper bound for each $x$. Comparing the suprema at $x$ and $x+\xi$ also gives
\begin{equation*}
 \Phi_R(x+\xi)-\Phi_R(x)\leq\sup_{y\in K_R}\xi\cdot y\,.
\end{equation*}
Consequently, $\Phi_R$ is finite and globally Lipschitz, and all its subgradients belong to $K_R$.

If $v\in\partial\Phi(x)\cap K_R$, then $\overline\Phi^*(v)=x\cdot v-\Phi(x)$, so~\eqref{eq:slopetruncation} gives $\Phi_R(x)=\Phi(x)$. Since the subgradients of $\Phi$ are locally bounded in $U$, $\Phi_R$ agrees with $\Phi$ on a neighborhood of every fixed compact subset of $U$ once $R$ is sufficiently large. At points where both functions are differentiable, $H(T)\leq R$ implies $T_R=T$. If $H(T)>R$, the inclusion $T_R\in K_R$ gives $H(T_R)\leq R<H(T)$. Therefore, along a sequence $R\to\infty$,
\begin{equation}\label{eq:domination}
 T_R\longrightarrow T\quad\text{a.e.},\qquad H(T_R)\leq H(T)\quad\text{a.e.}\,.
\end{equation}

By convolving $\Phi_R$ on $\R^n$ with a nonnegative mollifier, we obtain a smooth convex function $\Phi_{R,\varepsilon}$. Its gradient $V=\nabla\Phi_{R,\varepsilon}$ takes values in $K_R$, since $K_R$ is closed and convex. Moreover, $|V|\leq c_HR$ and $\diver V\geq0$. Since $V-e\in\Omega$,~\eqref{eq:normal} gives $(V-e)\cdot\nu\leq0$ almost everywhere on the boundary. Choose a smooth cutoff $0\leq\chi_N\leq1$ that equals one on $B_N$, vanishes outside $B_{2N}$, and satisfies $|\nabla\chi_N|\leq c_\chi/N$. By the Gauss--Green formula and~\eqref{eq:projected}, we obtain
\begin{equation}\label{eq:cutoffgreen}
 \int_U\chi_Nh\Delta\Phi_{R,\varepsilon}\de x+\int_{\R^{n-1}}\chi_N\Tr h\de x'
 \leq-\int_U\chi_N\nabla h\cdot V\de x-\int_Uh\nabla\chi_N\cdot V\de x\,.
\end{equation}
In the boundary integral, $\chi_N$ means $\chi_N(x',\varphi(x'))$. To justify the Gauss--Green formula on this Lipschitz graph, use the change of variables $\Xi(z',t)=(z',t+\varphi(z'))$, $t>0$. For a smooth vector field $F$, write $F'=(F_1,\ldots,F_{n-1})$. The transformed field has components $F_i\circ\Xi$ for $i<n$ and $(F_n-\nabla\varphi\cdot F')\circ\Xi$ for $i=n$. Its divergence is $(\diver F)\circ\Xi$, and integration on the half-space $t>0$ gives the boundary flux $(\nabla\varphi\cdot F'-F_n)\de x'$. By approximation in $W^{1,1}$ on bounded boundary patches, the same formula holds for $F=\chi_NhV$. This argument uses only the first weak derivatives of $\varphi$.

Set $A_N=U\cap(B_{2N}\setminus B_N)$. The absolute value of the last term in~\eqref{eq:cutoffgreen} is bounded by $c_HRN^{-1}\int_{A_N}h$. For fixed $R$ and $\varepsilon$, this bound tends to zero as $N\to\infty$ if $h\in L^1(U)$. If $h\in L^{n/(n-1)}(U)$, H\"older's inequality gives
\begin{equation*}
 \frac{c_HR}{N}\int_{A_N}h\de x\leq c_{n,H}R\|h\|_{L^{n/(n-1)}(A_N)}\longrightarrow0\,.
\end{equation*}
Let $N\to\infty$. Fatou's lemma applies to the nonnegative Laplacian term, and dominated convergence applies to the other terms. Next let $\varepsilon\downarrow0$. The gradients are uniformly bounded and converge almost everywhere to $T_R$, while the Laplacian measures converge weakly on compact subsets of $U$ to $\Delta_{\mathcal D'}\Phi_R$. To pass to the limit in the Laplacian term, test against $\zeta h$, where $\zeta\in C_c(U)$ and $0\leq\zeta\leq1$. This is a continuous function with compact support in $U$. Taking a sequence of such cutoffs increasing pointwise to one yields
\begin{equation*}
 \int_Uh\de(\Delta_{\mathcal D'}\Phi_R)+\int_{\R^{n-1}}\Tr h\de x'\leq-\int_U\nabla h\cdot T_R\de x\,.
\end{equation*}
Finally, let $R\to\infty$. The right-hand side converges by dominated convergence, using~\eqref{eq:pairassumption} and~\eqref{eq:domination}. On each compact subset of $U$, $\Phi_R=\Phi$ in a neighborhood for all sufficiently large $R$, so their Hessian measures agree there. Applying this observation on an increasing sequence of compact subsets of $U$ proves~\eqref{eq:green} for the full Laplacian measure, including its singular part.
\end{proof}

\section{Proof of the main theorem}\label{sec:proof}

The translated cone $D_1$ has positive distance from the origin. Indeed, for $y\in D_1$, we have $1\leq y_n-\varphi(y')\leq|y_n|+K|y'|\leq\sqrt{1+K^2}\,|y|$. Hence $\dist(0,D_1)\geq(1+K^2)^{-1/2}$. Choose $C>0$ so that
\begin{equation}\label{eq:model}
 W(y)=\frac C{p'}H(y)^{p'},\qquad \int_{D_1}W^{-a}\de y=1,\qquad B_0=\int_{D_1}W^{1-a}\de y\,.
\end{equation}
Since $D_1$ has positive distance from the origin and $p'(a-1)\geq p'(n-1)>n$, the integrals in~\eqref{eq:model} and $\int_{D_1}H(y)^{p'}W(y)^{-a}\de y$ are finite. The convex conjugate of $W$ is
\begin{equation}\label{eq:conjugate}
 W^*(\xi)=\frac{C^{1-p}}pH^\circ(\xi)^p\,.
\end{equation}
To verify this formula, write $t=H(y)$ in the supremum defining $W^*$ and maximize $tH^\circ(\xi)-Ct^{p'}/p'$ over $t\geq0$. Set
\begin{equation}\label{eq:deficit}
 \begin{gathered}
 K_0=\frac{a-1}{p}C^{1-p}r^p,\qquad
 \mathcal V_a(f)=\begin{cases}(a-n)\int_U f^q\de x,&a>n,\\0,&a=n,\end{cases}\\
 \mathcal D_a(f)=\mathcal V_a(f)+K_0E(f)-J(f)-B_0\,.
 \end{gathered}
\end{equation}
At $a=n$, we define $\mathcal V_a(f)=0$ directly, since $f^q$ need not be integrable over $U$.

\begin{proposition}\label{prop:core}
Let $f\in X_a$ be a nonzero nonnegative extremal of~\eqref{eq:sharp}, normalized by $\int_Uf^s\de x=1$. Then there is a finite convex function $\Phi$ on $U$ such that $T=\nabla\Phi$ transports the probability measure $f^s\de x$ to $W^{-a}\mathbf1_{D_1}\de y$. This means that
\begin{equation*}
 \int_U\psi(T(x))f(x)^s\de x=\int_{D_1}\psi(y)W(y)^{-a}\de y
\end{equation*}
for every bounded Borel function $\psi$. Let $D_A^2\Phi$ denote the Alexandrov Hessian, namely the matrix of second derivatives in the almost-everywhere second-order expansion of $\Phi$. Write
\begin{equation*}
 B=D_A^2\Phi,\qquad \Delta_{\mathcal D'}\Phi=(\tr B)\de x+\mu_s\,,
\end{equation*}
where $\mu_s$ is the singular part of the distributional Laplacian. Then
\begin{equation}\label{eq:core}
 \mathcal D_a(f)\geq
 \int_U f^q\big[a-n+\tr B-a(\det B)^{1/a}\big]\de x
 +\int_U f^q\de\mu_s\geq0\,.
\end{equation}
At $a=n$, the term $a-n$ in the integrand vanishes.
\end{proposition}

\begin{proof}
By Lemma~\ref{lem:regularity}, $f$ has a strictly positive continuous representative in $U$. Define $g=f^{-r}$, $\rho=f^s=g^{-a}$ and $h=f^q=g^{1-a}$. The positive lower bound for $f$ on compact subsets allows us to apply the Sobolev chain rule, giving
\begin{equation}\label{eq:gchain}
 \nabla h=-(a-1)\rho\nabla g,\qquad \rho H^\circ(\nabla g)^p=r^pH^\circ(\nabla f)^p\,.
\end{equation}
McCann's theorem~\cite[Main Theorem, p.\ 310]{McC} gives a convex potential $\Phi$ whose gradient transports $\rho\de x$ to $W^{-a}\mathbf1_{D_1}\de y$ as stated above. The theorem does not require moment bounds on these measures. Both measures are absolutely continuous with respect to Lebesgue measure. Since $\rho>0$ throughout $U$, $\Phi$ is finite on a subset of $U$ whose complement has Lebesgue measure zero. The set where $\Phi$ is finite is convex, and every point of $U$ lies in the interior of the convex hull of that subset. Thus $\Phi$ is finite throughout $U$. McCann's Monge--Amp\`ere identity for these measures~\cite[Theorem 4.4 and Remark 4.5]{McC97} gives
\begin{equation}\label{eq:jacobian}
 \rho=W(T)^{-a}\det B,\qquad W(T)=g(\det B)^{1/a}\quad\text{a.e. in }U\,.
\end{equation}
Here $B$ is the Alexandrov Hessian; the identity holds without assuming that $T$ is smooth or that the measures have finite second moments. Since $\rho>0$ and $W(T)>0$, it also gives $\det B>0$ almost everywhere. By the transport identity in the proposition and~\eqref{eq:model}, we obtain
\begin{equation}\label{eq:moments}
 \int_U\rho W(T)\de x=B_0=\int_Uh(\det B)^{1/a}\de x\,.
\end{equation}
The distributional Hessian $D^2_{\mathcal D'}\Phi$ is a positive semidefinite matrix-valued measure. Its absolutely continuous part is $B\de x$, and the trace of its singular part is the nonnegative measure $\mu_s$.

H\"older's inequality and~\eqref{eq:chainexponents} give
\begin{align}\label{eq:pairing}
 \int_UH^\circ(\nabla h)H(T)\de x
 &\leq qE(f)^{1/p}\left(\int_Uf^{(q-1)p'}H(T)^{p'}\de x\right)^{1/p'}\notag\\
 &=qE(f)^{1/p}\left(\int_{D_1}H(y)^{p'}W(y)^{-a}\de y\right)^{1/p'}<\infty\,.
\end{align}
By~\eqref{eq:pairing}, Lemma~\ref{lem:green} applies. Use its boundary integration inequality~\eqref{eq:green} and the chain rule~\eqref{eq:gchain}, followed by Fenchel's inequality $\nabla g\cdot T\leq W(T)+W^*(\nabla g)$. The formula~\eqref{eq:conjugate} then gives
\begin{align}\label{eq:fenchelgreen}
 \int_Uh\tr B\de x+\int_Uh\de\mu_s+J(f)
 &\leq(a-1)\int_U\rho\nabla g\cdot T\de x\notag\\
 &\leq(a-1)B_0+K_0E(f)\,.
\end{align}
All terms are finite by~\eqref{eq:gchain},~\eqref{eq:moments} and~\eqref{eq:pairing}. The arithmetic--geometric mean inequality for the eigenvalues of $B$, followed by $t^{n/a}\leq1+(n/a)(t-1)$ for $t\geq0$, gives
\begin{equation}\label{eq:matrix}
 (\det B)^{1/a}\leq\left(\frac{\tr B}{n}\right)^{n/a}\leq\frac{a-n+\tr B}{a}\,.
\end{equation}
Combining~\eqref{eq:moments},~\eqref{eq:fenchelgreen} and~\eqref{eq:matrix} proves~\eqref{eq:core}. When $a=n$, the term $(a-n)h$ is absent, so the argument does not require $h\in L^1(U)$.
\end{proof}

The next lemma uses the multiplication by constants and spatial dilations from Zugmeyer's proof~\cite[Section 3.2, equations (16)--(17)]{Zug}. We track the equality conditions in these steps and determine the constant by evaluating the expressions at the explicit extremal~\eqref{eq:zug-model}.

\begin{lemma}\label{lem:normalization}
Let $F\in X_a$ be a nonzero nonnegative extremal of~\eqref{eq:sharp}. There are constants $c>0$ and $\lambda>0$ such that $f(x)=cF(\lambda x)$ satisfies $\int_Uf^s\de x=1$ and $\mathcal D_a(f)=0$. When $a=n$, we may take $\lambda=1$.
\end{lemma}

\begin{proof}
Write $M(F)=\int_UF^s\de x$ and $v=M(F)^{-1/s}F$. Multiplication by this positive constant preserves equality in~\eqref{eq:sharp}. Define
\begin{equation}\label{eq:scalar-constants}
 \gamma=\frac{p-1}{a},\qquad\eta=\frac{a-1}{a},\qquad
 b=\frac{a-1}{a-p},\qquad b'=\frac{a-1}{p-1},\qquad
 D_0=\frac{K_0^b}{b(B_0b')^{b-1}}\,.
\end{equation}
The constants in~\eqref{eq:scalar-constants} satisfy $1/b+1/b'=1$, $\eta=b'\gamma$, $q/s=\eta$ and $\eta-p/s=\gamma$. Applying Proposition~\ref{prop:core} to $v$ and using the powers of $M(F)$ in each term, we obtain
\begin{align}\label{eq:normalchain}
 0\leq M(F)^\eta\mathcal D_a(v)
 &=K_0E(F)M(F)^\gamma-B_0M(F)^\eta+\mathcal V_a(F)-J(F)\notag\\
 &\leq D_0E(F)^b+\mathcal V_a(F)-J(F)\,.
\end{align}
For the last inequality, set $z=M(F)^\gamma$ and use the maximum of $K_0E(F)z-B_0z^{b'}$ over $z\geq0$. This maximum equals $D_0E(F)^b$ and is attained at $z=(K_0E(F)/(B_0b'))^{b-1}$.

For $a>n$ and $\lambda>0$, set $F_\lambda(x)=F(\lambda x)$. Since $U$ is a cone, this function is defined on the same domain. A change of variables gives
\begin{equation*}
 E(F_\lambda)=\lambda^{p-n}E(F),\qquad
 I(F_\lambda)=\lambda^{-n}I(F),\qquad
 J(F_\lambda)=\lambda^{-(n-1)}J(F)\,.
\end{equation*}
These identities and~\eqref{eq:parameters} show that $F_\lambda$ is also an extremal. Apply~\eqref{eq:normalchain} to $F_\lambda$ and multiply its last expression by $\lambda^{n-1}$. The result is the following function of $\lambda$:
\begin{equation}\label{eq:scale-function}
 D_0E(F)^b\lambda^d+\frac{(a-n)I(F)}\lambda-J(F),\qquad d=\frac{(p-1)(a-n)}{a-p}>0\,.
\end{equation}
The minimum in~\eqref{eq:scale-function} occurs at
\begin{equation}\label{eq:scale}
 \lambda_* =\left(\frac{(a-n)I(F)}{dD_0E(F)^b}\right)^{1/(d+1)}\,.
\end{equation}
The minimum of the first two terms in~\eqref{eq:scale-function} is $\mathcal C E(F)^{b/(d+1)}I(F)^{d/(d+1)}$, where
\begin{equation}\label{eq:constant}
 \mathcal C=(d+1)d^{-d/(d+1)}D_0^{1/(d+1)}(a-n)^{d/(d+1)},\quad
 \frac b{d+1}=\frac{q\theta}p,\quad\frac d{d+1}=1-\theta\,.
\end{equation}
At $a=n$, no minimization over $\lambda$ is needed: $b=q/p$, and the last line of~\eqref{eq:normalchain} equals $D_0E(F)^{q/p}-J(F)$.

Consider $u_0(x)=W(x+e)^{-1/r}$. It belongs to $X_a$, satisfies $M(u_0)=1$, and is a positive multiple of Zugmeyer's extremal~\eqref{eq:zug-model}. Hence it attains equality in~\eqref{eq:sharp} by~\cite[Theorem 1.2]{Zug}. At every differentiability point $y\ne0$ of $H$, convexity and Euler's identity give
\begin{equation*}
 H(z)\geq H(y)+\nabla H(y)\cdot(z-y)=\nabla H(y)\cdot z
 \quad(z\in\R^n),\qquad y\cdot\nabla H(y)=H(y)\,.
\end{equation*}
Taking the supremum over $H(z)\leq1$, with equality at $z=y/H(y)$, yields $H^\circ(\nabla H(y))=1$. We obtain the identities
\begin{equation}\label{eq:calibration}
 \int_Uu_0^q\de x=B_0,\qquad K_0E(u_0)=b'B_0,\qquad J(u_0)=\big[(a-1)p'-n\big]B_0\,.
\end{equation}
The first identity follows by the change of variables $y=x+e$. For the energy identity, apply~\eqref{eq:gchain} with $g=W(x+e)$ and use $W^*(\nabla W)=(p'-1)W$. For the trace identity, let $h_0=W(x+e)^{1-a}$. Then
\begin{equation}\label{eq:model-divergence}
 \diver\big(h_0(x)(x+e)\big)=\big[n-(a-1)p'\big]h_0(x)\quad\text{a.e.}\,.
\end{equation}
Apply the Gauss--Green formula to $\chi_Nh_0(x)(x+e)$, where $\chi_N$ is the cutoff used in the proof of Lemma~\ref{lem:green}. On the graph, $(x+e)\cdot\nu=e\cdot\nu$, so~\eqref{eq:projected} and~\eqref{eq:model-divergence} give the trace identity in~\eqref{eq:calibration}. The error term containing $\nabla\chi_N$ is bounded by a constant times $\int_{U\cap(B_{2N}\setminus B_N)}h_0\de x$, which tends to zero because $h_0\in L^1(U)$ and $|x+e|/N$ is bounded on this annulus. This argument uses only the local Lipschitz continuity of $H$.

Since $K_0E(u_0)=b'B_0$, the value $z=M(u_0)^\gamma=1$ attains the scalar maximum used in~\eqref{eq:normalchain}. Thus $D_0E(u_0)^b=(b'-1)B_0$. Substituting~\eqref{eq:calibration} into~\eqref{eq:deficit} also gives $\mathcal D_a(u_0)=0$. If $a>n$, then $dD_0E(u_0)^b=(a-n)B_0$, so the minimizing value in~\eqref{eq:scale} is $\lambda_*=1$ for $u_0$. Therefore, evaluating the minimum at $u_0$ gives
\begin{equation*}
 \mathcal C=\frac{J(u_0)}{E(u_0)^{q\theta/p}I(u_0)^{1-\theta}}=D^q\,.
\end{equation*}
The last equality holds because $u_0$ attains equality in~\eqref{eq:sharp}. At $a=n$, the same calculation gives $D_0=J(u_0)/E(u_0)^{q/p}=D^q$.

For $a>n$, apply~\eqref{eq:normalchain} to $F_{\lambda_*}$ and multiply the last expression by $\lambda_*^{n-1}$. By~\eqref{eq:constant} and $\mathcal C=D^q$, the result is $D^qE(F)^{q\theta/p}I(F)^{1-\theta}-J(F)$, which is zero because $F$ is an extremal. At $a=n$, the last expression in~\eqref{eq:normalchain} is $D^qE(F)^{q/p}-J(F)=0$. Thus both inequalities in~\eqref{eq:normalchain} are equalities. For $a>n$, take $\lambda_*$ from~\eqref{eq:scale} and set $f=M(F_{\lambda_*})^{-1/s}F_{\lambda_*}$. For $a=n$, set $f=v$. In either case, $\int_Uf^s\de x=1$ and $\mathcal D_a(f)=0$.
\end{proof}

\begin{proof}[Proof of Theorem~\ref{thm:classification}]
Let $F\in X_a$ be a nonzero real-valued extremal and put $G=|F|$. Since
$E(G)=E(F)$, $J(G)=J(F)$ and, for $a>n$, $I(G)=I(F)$,
the function $G$ is a nonnegative extremal. Apply Lemma~\ref{lem:normalization} to $G$ and denote the resulting function by $f$. Since $\mathcal D_a(f)=0$, both nonnegative integrals in~\eqref{eq:core} vanish. If $a>n$, the scalar inequality
\begin{equation*}
 t^{n/a}\leq1+\frac na(t-1),\qquad t\geq0,
\end{equation*}
is an equality only at $t=1$. Taking $t=\tr B/n$ in~\eqref{eq:matrix} therefore gives $\tr B=n$ almost everywhere. Equality in the arithmetic--geometric mean inequality for the eigenvalues of $B$ then gives $B=I_n$ almost everywhere. If $a=n$, the scalar inequality is an identity, and equality in the arithmetic--geometric mean inequality gives $B=\sigma(x)I_n$. Here $\sigma(x)>0$ almost everywhere by~\eqref{eq:jacobian}.

In either case, $\int_Uf^q\de\mu_s=0$. For each compact $\mathcal K\Subset U$, Lemma~\ref{lem:regularity} gives $f\geq c_{\mathcal K}>0$, so $\mu_s(\mathcal K)=0$ and hence $\mu_s=0$. Let $S$ be the singular part of the distributional Hessian. This is a positive semidefinite matrix-valued measure with trace $\mu_s=0$, so all diagonal measures $S_{ii}$ vanish. For the standard coordinate vectors $e_i,e_j$, positivity of $(e_i\pm e_j)^TS(e_i\pm e_j)=\pm2S_{ij}$ forces $S_{ij}=0$. Thus the entire singular part of the Hessian is zero.

For $a>n$ this proves $D^2_{\mathcal D'}\Phi=I_n\de x$. At $a=n$, $\partial_{ij}\Phi=\delta_{ij}\sigma$ in distributions. Given $j$, choose $i\ne j$, which is possible because $n\geq2$. Commuting distributional derivatives gives
\begin{equation*}
 \partial_j\sigma=\partial_j\partial_{ii}\Phi=\partial_i\partial_{ij}\Phi=0\,.
\end{equation*}
Since $U$ is connected, $\sigma$ is a positive constant $\sigma_0$. In both cases, $D^2_{\mathcal D'}\Phi=\sigma_0I_n\de x$, with $\sigma_0=1$ when $a>n$. The distribution $\Phi-\sigma_0|x|^2/2$ has zero Hessian and is therefore affine. Taking its gradient gives
\begin{equation}\label{eq:affine}
 T(x)=\sigma_0x+b_0,\qquad\sigma_0>0,\qquad\sigma_0=1\ \text{if }a>n\,.
\end{equation}

Since $f>0$ in $U$, the measure $f^s\de x$ has support $\overline U=\Omega$. The measure $W^{-a}\mathbf1_{D_1}\de y$ has support $D_1$. The affine map $x\mapsto\sigma_0x+b_0$ is a homeomorphism of $\R^n$, so the measure obtained by transporting $f^s\de x$ under this map has support $\sigma_0\Omega+b_0$. Comparing the supports and using $\sigma_0\Omega=\Omega$, we obtain
\begin{equation}\label{eq:support-match}
 \sigma_0\Omega+b_0=\Omega+e,\qquad \Omega+(b_0-e)=\Omega\,.
\end{equation}
We now show that $\Omega+\ell=\Omega$ if and only if $\ell\in L$. If $\Omega+\ell=\Omega$, the fact that $0\in\Omega$ gives both $\ell\in\Omega$ and $-\ell\in\Omega$. Conversely, if $\ell,-\ell\in\Omega$, closure of a convex cone under addition gives $\Omega+\ell\subset\Omega$ and $\Omega-\ell\subset\Omega$. Translating the second inclusion by $\ell$ gives $\Omega\subset\Omega+\ell$. These cone properties also imply that $L=\Omega\cap(-\Omega)$ is a linear subspace. Therefore~\eqref{eq:support-match} yields $b_0=e+\ell_0$ for some $\ell_0\in L$.

Substituting~\eqref{eq:affine} into~\eqref{eq:jacobian} and using $\det B=\sigma_0^n$, we determine $g$:
\begin{equation}\label{eq:normalized-profile}
 g(x)=\begin{cases}W(x+e+\ell_0),&a>n,\\
 \sigma_0^{-1}W(\sigma_0x+e+\ell_0),&a=n.\end{cases}
\end{equation}
Since $f=g^{-1/r}$ and $W$ is homogeneous of degree $p'$,~\eqref{eq:normalized-profile} shows that $f$ is a positive multiple of $H(x+e+\ell_0)^{-(a-p)/(p-1)}$ when $a>n$. When $a=n$, it is a positive multiple of $H(x+\sigma_0^{-1}e+\sigma_0^{-1}\ell_0)^{-(n-p)/(p-1)}$. To recover $G$ when $a>n$, use $f(x)=M(G_{\lambda_*})^{-1/s}G(\lambda_*x)$ and set $z=\lambda_*x$. This gives
\begin{equation}\label{eq:restoration}
 G(z)=\widetilde A_+\,H(z/\lambda_*+e+\ell_0)^{-\frac{a-p}{p-1}}
      =A_+ H(z+\lambda_*e+\lambda_*\ell_0)^{-\frac{a-p}{p-1}}\,,
 \qquad A_+=\widetilde A_+\lambda_*^{\frac{a-p}{p-1}}>0\,.
\end{equation}
Thus~\eqref{eq:family} holds for $G$ with $t=\lambda_*$ and $\ell=\lambda_*\ell_0$ when $a>n$. When $a=n$, it holds with $t=\sigma_0^{-1}$ and $\ell=\sigma_0^{-1}\ell_0$. Since $L$ is a linear subspace, $\ell\in L$ in either case. Therefore $G$ has the form~\eqref{eq:family} with a positive multiplicative constant $A_+$.

It remains to determine the sign of $F$. Fix a relatively compact open set $V_0\Subset U$. The formula for $G$ gives $|F|=G\geq c_{V_0}>0$ almost everywhere on $V_0$. Choose a Lipschitz function $\chi:\R\to\R$ that agrees with $\operatorname{sgn}$ outside $(-c_{V_0}/2,c_{V_0}/2)$. Then $\chi(F)=\operatorname{sgn}F$ and $\chi'(F)=0$ almost everywhere on $V_0$. By the Sobolev chain rule, $\operatorname{sgn}F\in W^{1,p}(V_0)$ and $\nabla(\operatorname{sgn}F)=0$ there. Since $V_0$ was arbitrary, these statements hold locally throughout $U$. Connectedness of $U$ implies that $\operatorname{sgn}F$ is a constant in $\{-1,1\}$ almost everywhere. Hence $F=\pm G$, proving~\eqref{eq:family} with $A\in\R\setminus\{0\}$.

Conversely, $f_0$ in~\eqref{eq:zug-model} belongs to $X_a$ and attains equality by Zugmeyer's theorem~\cite[Theorem 1.2]{Zug}. Multiplying a function by a positive constant or replacing $f(x)$ by $f(\lambda x)$, with $\lambda>0$, preserves equality in~\eqref{eq:sharp}. Translation by $\ell=(\ell',\ell_n)\in L$ also preserves equality. Indeed, $\Omega+\ell=\Omega$ implies
\begin{equation*}
 \varphi(x'+\ell')=\varphi(x')+\ell_n\,.
\end{equation*}
Hence translating a boundary point by $\ell$ corresponds to the change of variables $x'\mapsto x'+\ell'$ and preserves $\de x'$. Because translation by $\ell$ also preserves $U$ and Lebesgue measure, it leaves $E$, $J$ and, when $a>n$, $I$ unchanged. Set $k=(a-p)/(p-1)$. Every function in~\eqref{eq:family} can be written as $At^{-k}f_0((x+\ell)/t)$, so it attains equality when $A>0$. The same is true when $A<0$, since all terms in~\eqref{eq:sharp} are unchanged by $f\mapsto-f$.
\end{proof}

\begingroup
\makeatletter
\let\@tocwrite\@gobbletwo
\makeatother
\section*{Declarations}

\noindent\textbf{Conflict of interest.}
The authors declare that they have no conflicts of interest.

\medskip

\noindent\textbf{Data availability.}
Data sharing is not applicable to this article, as no datasets were generated or analyzed during the current study.

\medskip

\noindent\textbf{Generative AI disclosure.}
The initial research ideas and overall direction of this work originated with the authors. Generative AI tools were used to assist with research discussions, checking mathematical arguments and references, and preparing and editing the manuscript, including its LaTeX source. The authors reviewed the AI-assisted material and take full responsibility for the mathematical results and the final manuscript.
\endgroup

\end{document}